\documentclass[letterpaper, 10 pt, conference]{ieeeconf}  

\usepackage{amsmath}
\usepackage{graphicx}
\usepackage{cite}
\usepackage{balance}
\usepackage{amssymb,amsfonts}

\usepackage{amssymb,amsfonts}

\usepackage{lipsum}
\usepackage{mathtools}
\usepackage{cuted}

\usepackage{balance}

\DeclareMathOperator{\tr}{tr}

\newtheorem{theorem}{\it Theorem}
\newtheorem{lemma}{\it Lemma}
\newtheorem{definition}{\it Definition}

\newtheorem{corollary}{\it Corollary}
\newtheorem{proposition}{\it Proposition}

\IEEEoverridecommandlockouts                              

\title{\LARGE \bf
	Independent Elliptical Distributions Minimize Their $\mathcal{W}_2$ Wasserstein Distance from Independent Elliptical Distributions\\ with the Same Density Generator
}

\author{Song Fang$^{1}$ and Quanyan Zhu$^{1}$
\thanks{$^{1}$ Song Fang and Quanyan Zhu are with the Department of Electrical and Computer Engineering, New York University, New York, USA
        {\tt\small song.fang@nyu.edu; quanyan.zhu@nyu.edu}}%
}

\begin{document}

\maketitle
\thispagestyle{empty}
\pagestyle{empty}

\begin{abstract}
	
	This short note is on a property of the $\mathcal{W}_2$ Wasserstein distance which indicates that independent elliptical distributions minimize their $\mathcal{W}_2$ Wasserstein distance from given independent elliptical distributions with the same density generators. Furthermore, we examine the implications of this property in the Gelbrich bound when the distributions are not necessarily elliptical. Meanwhile, we also generalize the results to the cases when the distributions are not independent. The primary purpose of this note is for the referencing of papers that need to make use of this property or its implications.
\end{abstract}


\section{Introduction}

%

%
%
%
%
%

The Wasserstein distance (see, e.g., \cite{peyre2019computational, panaretos2020invitation} and the references therein) is an important metric from optimal transport theory (see, e.g., \cite{villani2003topics, villani2008optimal, santambrogio2015optimal} and the references therein). 
In this short note, we focus on a property of the $\mathcal{W}_2$ Wasserstein distance which indicates that independent elliptical distributions minimize their $\mathcal{W}_2$ Wasserstein distance from given independent elliptical distributions with the same density generators; see Theorem~\ref{t1}. Moreover, we investigate what this property  implicates in the Gelbrich bound  for the $\mathcal{W}_2$ Wasserstein distance when the distributions are not necessarily elliptical; see Corollary~\ref{c1}. We then generalize the results in Theorem~\ref{t1} and Corollary~\ref{c1} to the cases when the distributions are not independent; see Theorem~\ref{t2} and Corollary~\ref{c2}, respectively. It is worth mentioning that, in a broad sense, the results presented in this note parallel those of \cite{KLproperties} (for the Kullback--Liebler divergence) as well as \cite{fang2017towards} (for entropy and mutual information).

\section{Preliminaries}

Throughout the note, we consider zero-mean real-valued continuous random variables and random vectors. We represent random variables and random vectors using boldface letters, e.g., $\mathbf{x}$, while the probability density function of $\mathbf{x}$ is denoted as $p_\mathbf{x}$.

The $\mathcal{W}_p$ Wasserstein distance (see, e.g., \cite{santambrogio2015optimal, peyre2019computational, panaretos2020invitation}) is defined as follows.

\begin{definition}
	The $\mathcal{W}_p$ (for $p \geq 1$) Wasserstein distance between distribution $p_{\mathbf{x}}$ and distribution $p_{\mathbf{y}}$ is defined as
	\begin{flalign}
	\mathcal{W}_p \left( p_{\mathbf{x}} ; p_{\mathbf{y}} \right)
	= \left( \inf_{ \mathbf{x}, \mathbf{y}} \mathbb{E} \left[ \left\| \mathbf{x} - \mathbf{y} \right\|^p \right] \right)^{\frac{1}{p}}, \nonumber
	\end{flalign}
	where $\mathbf{x}$ and $\mathbf{y}$ denote $m$-dimensional random vectors with distributions $p_{\mathbf{x}}$ and $p_{\mathbf{y}}$, respectively.
\end{definition}

Particularly when $p = 2$, the $\mathcal{W}_2$ distance is given by 
\begin{flalign}
\mathcal{W}_2 \left( p_{\mathbf{x}} ; p_{\mathbf{y}} \right)
= \sqrt{ \inf_{ \mathbf{x}, \mathbf{y}} \mathbb{E} \left[ \left\| \mathbf{x} - \mathbf{y} \right\|^2 \right] }. \nonumber
\end{flalign}

The following lemma (see, e.g., \cite{santambrogio2015optimal, peyre2019computational, panaretos2020invitation}) provides an explicit expression for the $\mathcal{W}_2$ distance between elliptical distributions with the same density generator. Note that Gaussian distributions are a special class of elliptical distributions (see, e.g., \cite{peyre2019computational}). Note also that hereinafter the random vectors are assumed to be zero-mean for simplicity.

\begin{lemma} \label{Gaussian}
	Consider $m$-dimensional elliptical random vectors $\mathbf{x}$ and $\mathbf{y}$ with the same density generator, while with covariance matrices $\Sigma_\mathbf{x}$ and $\Sigma_\mathbf{y}$, respectively.
	The $\mathcal{W}_2$ distance between distribution $p_\mathbf{x}$
    and distribution $p_\mathbf{y}$ is given by
	\begin{flalign}
	\mathcal{W}_2 \left( p_{\mathbf{x}} ; p_{\mathbf{y}} \right) 
	= \sqrt{ \tr \left[ \Sigma_{\mathbf{x}} + \Sigma_{\mathbf{y}} - 2 \left( \Sigma_{\mathbf{x}}^{\frac{1}{2}}  \Sigma_{\mathbf{y}}  \Sigma_{\mathbf{x}}^{\frac{1}{2}} \right)^{\frac{1}{2}} \right] }. \nonumber 
	\end{flalign}
\end{lemma}

\vspace{3mm}

Meanwhile, the Gelbrich bound (see, e.g., \cite{peyre2019computational, panaretos2020invitation}) is given as follows, which provides a generic lower bound for the $\mathcal{W}_2$ distance between distributions that are not necessarily elliptical.

\begin{lemma} \label{Gelbrich}
	Consider $m$-dimensional random vectors $\mathbf{x}$ and $\mathbf{y}$ with covariance matrices $\Sigma_\mathbf{x}$ and $\Sigma_\mathbf{y}$, respectively.
	The $\mathcal{W}_2$ distance between distribution $p_\mathbf{x}$
	and distribution $p_\mathbf{y}$ is lower bounded by
	\begin{flalign}
	\mathcal{W}_2 \left( p_{\mathbf{x}} ; p_{\mathbf{y}} \right) 
	\geq \sqrt{ \tr \left[ \Sigma_{\mathbf{x}} + \Sigma_{\mathbf{y}} - 2 \left( \Sigma_{\mathbf{x}}^{\frac{1}{2}}  \Sigma_{\mathbf{y}}  \Sigma_{\mathbf{x}}^{\frac{1}{2}} \right)^{\frac{1}{2}} \right] }. \nonumber
	\end{flalign}
\end{lemma}

\vspace{3mm}

%

\section{$\mathcal{W}_2$ Distance Minimizing Distributions}

We first present the following proposition.

\begin{proposition} \label{half}
	Consider a positive definite matrix $\Sigma_{\mathbf{x}} \in \mathbb{R}^{m \times m}$. Denote the diagonal terms of $\Sigma_{\mathbf{x}}$ by $\sigma_{\mathbf{x} \left( 1 \right)}^2, \ldots, \sigma_{\mathbf{x} \left( m \right)}^2$, while denote the eigenvalues of $\Sigma_{\mathbf{x}}$ by $\lambda_{1}, \ldots, \lambda_{m}$. Then, 
	\begin{flalign}
	\tr \left( \Sigma_{\mathbf{x}}^{\frac{1}{2}} \right)
	= \sum_{i=1}^{m} \lambda_i^{\frac{1}{2}} \leq \sum_{i=1}^{m}  \left[ \sigma_{\mathbf{x} \left( i \right)}^2  \right]^{\frac{1}{2}},
	\end{flalign} 
	where equality holds if and only if $\Sigma_{\mathbf{x}}$ is  a diagonal matrix.  
\end{proposition}

\begin{proof}
To begin with, denote 
\begin{flalign}
\Sigma_{\mathbf{x}} = 
\begin{bmatrix} \sigma_{\mathbf{x} \left( 1 \right)}^2 & \sigma_{\mathbf{x} \left( 1 \right) \mathbf{x} \left( 2 \right)}^2 & \cdots & \sigma_{\mathbf{x} \left( 1 \right) \mathbf{x} \left( m \right)}^2 \\
\sigma_{\mathbf{x} \left( 1 \right)\mathbf{x} \left( 2 \right)}^2 & \sigma_{\mathbf{x} \left( 2 \right)}^2 & \cdots & \sigma_{\mathbf{x} \left( 2 \right)\mathbf{x} \left( m \right)}^2 \\
\vdots & \vdots & \ddots & \vdots \\
\sigma_{\mathbf{x} \left( 1 \right)\mathbf{x} \left( m \right)}^2 & \sigma_{\mathbf{x} \left( 2 \right) \mathbf{x} \left( m \right)}^2 & \cdots & \sigma_{\mathbf{x} \left( m \right)}^2 
\end{bmatrix}
,\nonumber
\end{flalign} 
and
\begin{flalign}
\Lambda_{\mathbf{x}} = 
\begin{bmatrix} \sigma_{\mathbf{x} \left( 1 \right)}^2 & 0 & \cdots & 0 \\
0 & \sigma_{\mathbf{x} \left( 2 \right)}^2 & \cdots & 0 \\
\vdots & \vdots & \ddots & \vdots \\
0 & 0 & \cdots & \sigma_{\mathbf{x} \left( m \right)}^2 
\end{bmatrix}
.\nonumber
\end{flalign}
It is clear that $\lambda_i > 0, i = 1, \ldots, m$ and $\sigma_{\mathbf{x} \left( i \right)}^2 > 0, i = 1, \ldots, m$, since $\Sigma_{\mathbf{x}}$ is positive definite. Meanwhile, the eigenvalues of $\Sigma_{\mathbf{x}}^{\frac{1}{2}}$ are given by $\lambda_{1}^{\frac{1}{2}}, \ldots, \lambda_{m}^{\frac{1}{2}}$, and thus 
\begin{flalign}
\tr \left( \Sigma_{\mathbf{x}}^{\frac{1}{2}} \right)
= \sum_{i=1}^{m} \lambda_i^{\frac{1}{2}}
. \nonumber
\end{flalign} 
On the other hand, we have
\begin{flalign}
\tr \left( \Lambda_{\mathbf{x}}^{\frac{1}{2}} \right)
= \sum_{i=1}^{m}  \left[ \sigma_{\mathbf{x} \left( i \right)}^2  \right]^{\frac{1}{2}}
. \nonumber
\end{flalign} 
It now suffices to prove that
\begin{flalign}
\tr \left( \Sigma_{\mathbf{x}}^{\frac{1}{2}} \right)
\leq 
\tr \left( \Lambda_{\mathbf{x}}^{\frac{1}{2}} \right)
, \nonumber
\end{flalign}
where equality holds if and only if $\Sigma_{\mathbf{x}}$ is  a diagonal matrix. To prove this, note first that according to Klein's inequality (see, e.g., \cite{ruelle1999statistical}), we have
\begin{flalign}
\tr \left[ \left( - \Sigma_{\mathbf{x}}^{\frac{1}{2}} \right)
- \left( -  \Lambda_{\mathbf{x}}^{\frac{1}{2}} \right) 
- \left( \Sigma_{\mathbf{x}} -  \Lambda_{\mathbf{x}} \right)  \frac{1}{2} \left( -  \Lambda_{\mathbf{x}}^{ - \frac{1}{2}} \right) \right] \geq 0
, \nonumber
\end{flalign}
or equivalently,
\begin{flalign}
\tr \left[ \left( - \Sigma_{\mathbf{x}}^{\frac{1}{2}} \right)
- \left( -  \Lambda_{\mathbf{x}}^{\frac{1}{2}} \right)  \right] \geq 
\tr \left[  \left( \Sigma_{\mathbf{x}} -  \Lambda_{\mathbf{x}} \right)  \frac{1}{2} \left( -  \Lambda_{\mathbf{x}}^{ - \frac{1}{2}} \right) \right]
, \nonumber
\end{flalign}
since $f \left( x \right) = - x^{\frac{1}{2}}$ is operator convex and $f' \left( x \right) = - \frac{1}{2} x^{- \frac{1}{2} }$; in addition, herein equality holds if and only if $\Sigma_{\mathbf{x}} = \Lambda_{\mathbf{x}}$, i.e., if and only if $\Sigma_{\mathbf{x}}$ is diagonal, since $f \left( x \right) = - x^{\frac{1}{2}}$ is strictly convex. Meanwhile, note that
\begin{flalign}
\Sigma_{\mathbf{x}} -  \Lambda_{\mathbf{x}} 
& = \begin{bmatrix} \sigma_{\mathbf{x} \left( 1 \right)}^2 & \sigma_{\mathbf{x} \left( 1 \right) \mathbf{x} \left( 2 \right)}^2 & \cdots & \sigma_{\mathbf{x} \left( 1 \right) \mathbf{x} \left( m \right)}^2 \\
\sigma_{\mathbf{x} \left( 1 \right)\mathbf{x} \left( 2 \right)}^2 & \sigma_{\mathbf{x} \left( 2 \right)}^2 & \cdots & \sigma_{\mathbf{x} \left( 2 \right)\mathbf{x} \left( m \right)}^2 \\
\vdots & \vdots & \ddots & \vdots \\
\sigma_{\mathbf{x} \left( 1 \right)\mathbf{x} \left( m \right)}^2 & \sigma_{\mathbf{x} \left( 2 \right) \mathbf{x} \left( m \right)}^2 & \cdots & \sigma_{\mathbf{x} \left( m \right)}^2 
\end{bmatrix} \nonumber \\
&~~~~ - 
\begin{bmatrix} \sigma_{\mathbf{x} \left( 1 \right)}^2 & 0 & \cdots & 0 \\
0 & \sigma_{\mathbf{x} \left( 2 \right)}^2 & \cdots & 0 \\
\vdots & \vdots & \ddots & \vdots \\
0 & 0 & \cdots & \sigma_{\mathbf{x} \left( m \right)}^2 
\end{bmatrix}
\nonumber \\
& = \begin{bmatrix} 0 & \sigma_{\mathbf{x} \left( 1 \right) \mathbf{x} \left( 2 \right)}^2 & \cdots & \sigma_{\mathbf{x} \left( 1 \right) \mathbf{x} \left( m \right)}^2 \\
\sigma_{\mathbf{x} \left( 1 \right)\mathbf{x} \left( 2 \right)}^2 & 0 & \cdots & \sigma_{\mathbf{x} \left( 2 \right)\mathbf{x} \left( m \right)}^2 \\
\vdots & \vdots & \ddots & \vdots \\
\sigma_{\mathbf{x} \left( 1 \right)\mathbf{x} \left( m \right)}^2 & \sigma_{\mathbf{x} \left( 2 \right) \mathbf{x} \left( m \right)}^2 & \cdots & 0 
\end{bmatrix}
.\nonumber
\end{flalign}
and
\begin{flalign}
\Lambda_{\mathbf{x}}^{- \frac{1}{2}} 
& = \begin{bmatrix} \frac{1}{\sigma_{\mathbf{x} \left( 1 \right)}} & 0 & \cdots & 0 \\
0 & \frac{1}{\sigma_{\mathbf{x} \left( 2 \right)}} & \cdots & 0 \\
\vdots & \vdots & \ddots & \vdots \\
0 & 0 & \cdots & \frac{1}{\sigma_{\mathbf{x} \left( m \right)}}
\end{bmatrix}
,\nonumber
\end{flalign}
where $\sigma_{\mathbf{x} \left( i \right)} = \sqrt{\sigma_{\mathbf{x} \left( i \right)}^2}, i = 1, \ldots, m$.
As a result,
\begin{flalign}
&\left( \Sigma_{\mathbf{x}} -  \Lambda_{\mathbf{x}} \right)  \frac{1}{2} \left( -  \Lambda_{\mathbf{x}}^{ - \frac{1}{2}} \right) = - \frac{1}{2} \left( \Sigma_{\mathbf{x}} -  \Lambda_{\mathbf{x}} \right)   \left(  \Lambda_{\mathbf{x}}^{ - \frac{1}{2}} \right)
\nonumber \\
&~~~~ = - \frac{1}{2} \begin{bmatrix} 0 & \sigma_{\mathbf{x} \left( 1 \right) \mathbf{x} \left( 2 \right)}^2 & \cdots & \sigma_{\mathbf{x} \left( 1 \right) \mathbf{x} \left( m \right)}^2 \\
\sigma_{\mathbf{x} \left( 1 \right)\mathbf{x} \left( 2 \right)}^2 & 0 & \cdots & \sigma_{\mathbf{x} \left( 2 \right)\mathbf{x} \left( m \right)}^2 \\
\vdots & \vdots & \ddots & \vdots \\
\sigma_{\mathbf{x} \left( 1 \right)\mathbf{x} \left( m \right)}^2 & \sigma_{\mathbf{x} \left( 2 \right) \mathbf{x} \left( m \right)}^2 & \cdots & 0 
\end{bmatrix} \nonumber \\
&~~~~~~~~ \times
\begin{bmatrix} \frac{1}{\sigma_{\mathbf{x} \left( 1 \right)}} & 0 & \cdots & 0 \\
0 & \frac{1}{\sigma_{\mathbf{x} \left( 2 \right)}} & \cdots & 0 \\
\vdots & \vdots & \ddots & \vdots \\
0 & 0 & \cdots & \frac{1}{\sigma_{\mathbf{x} \left( m \right)}}
\end{bmatrix} \nonumber \\
&~~~~ = - \frac{1}{2} \begin{bmatrix} 0 & \frac{ \sigma_{\mathbf{x} \left( 1 \right) \mathbf{x} \left( 2 \right)}^2} {\sigma_{\mathbf{x} \left( 2 \right)}} & \cdots & \frac{ \sigma_{\mathbf{x} \left( 1 \right) \mathbf{x} \left( m \right)}^2 }{\sigma_{\mathbf{x} \left( m \right)}} \\
\frac{ \sigma_{\mathbf{x} \left( 1 \right)\mathbf{x} \left( 2 \right)}^2 } {\sigma_{\mathbf{x} \left( 1 \right)}} & 0 & \cdots & \frac{ \sigma_{\mathbf{x} \left( 2 \right)\mathbf{x} \left( m \right)}^2 } {\sigma_{\mathbf{x} \left( m \right)}} \\
\vdots & \vdots & \ddots & \vdots \\
\frac{ \sigma_{\mathbf{x} \left( 1 \right)\mathbf{x} \left( m \right)}^2 } {\sigma_{\mathbf{x} \left( 1 \right)}} & \frac{ \sigma_{\mathbf{x} \left( 2 \right) \mathbf{x} \left( m \right)}^2 }{\sigma_{\mathbf{x} \left( 2 \right)}} & \cdots & 0 
\end{bmatrix} 
,\nonumber
\end{flalign}
and hence
\begin{flalign}
\tr \left[ \left( \Sigma_{\mathbf{x}}
- \Lambda_{\mathbf{x}} \right) \frac{1}{2} \left( -  \Lambda_{\mathbf{x}}^{ - \frac{1}{2}} \right) \right]
= 0
. \nonumber
\end{flalign}
Consequently,
\begin{flalign}
\tr \left( \Lambda_{\mathbf{x}}^{\frac{1}{2}} \right) - \tr \left( \Sigma_{\mathbf{x}}^{\frac{1}{2}} \right)
&= \tr \left( \Lambda_{\mathbf{x}}^{\frac{1}{2}} - \Sigma_{\mathbf{x}}^{\frac{1}{2}} \right) \nonumber \\
&\geq \tr \left[ \left( \Sigma_{\mathbf{x}}
- \Lambda_{\mathbf{x}} \right) \frac{1}{2} \left( -  \Lambda_{\mathbf{x}}^{ - \frac{1}{2}} \right) \right]
= 0
, \nonumber
\end{flalign}
where equality holds if and only if $\Sigma_{\mathbf{x}}$ is diagonal.
\end{proof}

As a matter of fact,  it can be proved more generally that for any $0 < q < 1$, 
\begin{flalign}
\tr \left( \Sigma_{\mathbf{x}}^{q} \right)
= \sum_{i=1}^{m} \lambda_i^{q} \leq \sum_{i=1}^{m}  \left[ \sigma_{\mathbf{x} \left( i \right)}^2  \right]^{q},
\end{flalign} 
where equality holds if and only if $\Sigma_{\mathbf{x}}$ is a diagonal matrix, in a similar spirit to the case of $q = \frac{1}{2}$. Although this note only concerns the case of $q = \frac{1}{2}$, the more general case of $0 < q < 1$ may be found useful in other settings.

We now proceed to present the main results of this note.

\begin{theorem} \label{t1}
	Consider $m$-dimensional random vectors $\mathbf{x}$ and $\mathbf{y}$ with positive definite covariance matrices $\Sigma_{\mathbf{x}}$ and $\Sigma_{\mathbf{y}}$, respectively. Suppose that $\mathbf{x}$ is elliptically distributed with density generator $g_\mathbf{x} \left( \mathbf{u} \right)$, whereas $\mathbf{y}$ is not necessarily elliptical. In addition, suppose that $\Sigma_{\mathbf{x}}$ is a diagonal matrix ($\mathbf{x}$ is independent element-wise), i.e.,
	\begin{flalign}
	\Sigma_{\mathbf{x}}
	= \Lambda_{\mathbf{x}}
	= \mathrm{diag} \left( \sigma_{\mathbf{x} \left( 1 \right)}^2, \ldots, \sigma_{\mathbf{x} \left( m \right)}^2 \right).
	\end{flalign}
	Meanwhile, denote the diagonal terms of $\Sigma_{\mathbf{y}}$ by $\sigma_{\mathbf{y} \left( 1 \right)}^2, \ldots, \sigma_{\mathbf{y} \left( m \right)}^2$, whereas $\Sigma_{\mathbf{y}}$ is not necessarily diagonal.
	Then,
	\begin{flalign} \label{b1}
	\mathcal{W}_2 \left( p_{\mathbf{x}} ; p_{\mathbf{y}} \right)
	& \geq \sqrt{\sum_{i=1}^{m} \left\{  \sigma_{\mathbf{x} \left( i \right)}^2 + \sigma_{\mathbf{y} \left( i \right)}^2 - 2 \left[ \sigma_{\mathbf{x} \left( i \right)}^2 \sigma_{\mathbf{y} \left( i \right)}^2 \right]^{\frac{1}{2}} \right\}} \nonumber \\
	& = \sqrt{\sum_{i=1}^{m} \left\{  \left[ \sigma_{\mathbf{x} \left( i \right)}^2  \right]^{\frac{1}{2}}
	- \left[  \sigma_{\mathbf{y} \left( i \right)}^2 \right]^{\frac{1}{2}} \right\}^2 },
	\end{flalign}
	where equality holds if $\mathbf{y}$ is elliptical with the same density generator as $\mathbf{x}$, i.e., $g_\mathbf{y} \left( \mathbf{u} \right) = g_\mathbf{x} \left( \mathbf{u} \right)$, while $\Sigma_{\mathbf{y}}$ is a diagonal matrix as
	\begin{flalign}
	\Sigma_{\mathbf{y}}
	= \Lambda_{\mathbf{y}}
	= \mathrm{diag} \left( \sigma_{\mathbf{y} \left( 1 \right)}^2, \ldots, \sigma_{\mathbf{y} \left( m \right)}^2 \right),
	\end{flalign}
	i.e., $\mathbf{y}$ is independent element-wise.
\end{theorem}


\begin{proof}
	It follows from Lemma~\ref{Gaussian} and Lemma~\ref{Gelbrich} that
	\begin{flalign}
	\mathcal{W}_2 \left( p_{\mathbf{x}} ; p_{\mathbf{y}} \right) 
    \geq \sqrt{ \tr \left[ \Sigma_{\mathbf{x}} + \Sigma_{\mathbf{y}} - 2 \left( \Sigma_{\mathbf{x}}^{\frac{1}{2}}  \Sigma_{\mathbf{y}}  \Sigma_{\mathbf{x}}^{\frac{1}{2}} \right)^{\frac{1}{2}} \right] } 
	,\nonumber
	\end{flalign}
	where equality holds if $\mathbf{y}$ is elliptical with the same density generator as $\mathbf{x}$.
	Note then that
	\begin{flalign} 
	& \sqrt{ \tr \left[ \Sigma_{\mathbf{x}} + \Sigma_{\mathbf{y}} - 2 \left( \Sigma_{\mathbf{x}}^{\frac{1}{2}}  \Sigma_{\mathbf{y}}  \Sigma_{\mathbf{x}}^{\frac{1}{2}} \right)^{\frac{1}{2}} \right] } \nonumber \\
	&~~~~ = \sqrt{ \tr \left[ \Lambda_{\mathbf{x}} + \Sigma_{\mathbf{y}} - 2 \left( \Lambda_{\mathbf{x}}^{\frac{1}{2}}  \Sigma_{\mathbf{y}}  \Lambda_{\mathbf{x}}^{\frac{1}{2}} \right)^{\frac{1}{2}} \right] } \nonumber \\
	&~~~~ = \sqrt{  \tr \left( \Lambda_{\mathbf{x}} \right) +  \tr \left(  \Sigma_{\mathbf{y}} \right) - 2  \tr \left[ \left( \Lambda_{\mathbf{x}}^{\frac{1}{2}}  \Sigma_{\mathbf{y}}  \Lambda_{\mathbf{x}}^{\frac{1}{2}} \right)^{\frac{1}{2}} \right]}
	,\nonumber
	\end{flalign}
	where
	\begin{flalign}
	\tr \left( \Lambda_{\mathbf{x}} \right) =
	\sum_{i=1}^{m}  \sigma_{\mathbf{x} \left( i \right)}^2, \nonumber
	\end{flalign}
	and
	\begin{flalign}
	\tr \left( \Sigma_{\mathbf{y}} \right) =
	\sum_{i=1}^{m}  \sigma_{\mathbf{y} \left( i \right)}^2. \nonumber
	\end{flalign}
	It then remains to prove that
	\begin{flalign}
	\tr \left[ \left( \Lambda_{\mathbf{x}}^{\frac{1}{2}}  \Sigma_{\mathbf{y}}  \Lambda_{\mathbf{x}}^{\frac{1}{2}} \right)^{\frac{1}{2}} \right] \leq \sum_{i=1}^{m}  \left[ \sigma_{\mathbf{x} \left( i \right)}^2 \sigma_{\mathbf{y} \left( i \right)}^2 \right]^{\frac{1}{2}}
	,\nonumber
	\end{flalign} 
	where equality holds if $\Sigma_{\mathbf{y}}$ is diagonal as
	\begin{flalign}
	\Sigma_{\mathbf{y}}
	= \Lambda_{\mathbf{y}}
	= \mathrm{diag} \left( \sigma_{\mathbf{y} \left( 1 \right)}^2, \ldots, \sigma_{\mathbf{y} \left( m \right)}^2 \right). \nonumber
	\end{flalign}
	To prove this, denote first that
	\begin{flalign}
	\Sigma_{\mathbf{y}} = 
	\begin{bmatrix} \sigma_{\mathbf{y} \left( 1 \right)}^2 & \sigma_{\mathbf{y} \left( 1 \right) \mathbf{y} \left( 2 \right)}^2 & \cdots & \sigma_{\mathbf{y} \left( 1 \right) \mathbf{y} \left( m \right)}^2 \\
	\sigma_{\mathbf{y} \left( 1 \right)\mathbf{y} \left( 2 \right)}^2 & \sigma_{\mathbf{y} \left( 2 \right)}^2 & \cdots & \sigma_{\mathbf{y} \left( 2 \right)\mathbf{y} \left( m \right)}^2 \\
	\vdots & \vdots & \ddots & \vdots \\
	\sigma_{\mathbf{y} \left( 1 \right)\mathbf{y} \left( m \right)}^2 & \sigma_{\mathbf{y} \left( 2 \right) \mathbf{y} \left( m \right)}^2 & \cdots & \sigma_{\mathbf{y} \left( m \right)}^2 
	\end{bmatrix}
	,\nonumber
	\end{flalign} 
	and accordingly,
	\begin{flalign}
	&\Lambda_{\mathbf{x}}^{\frac{1}{2}}  \Sigma_{\mathbf{y}}  \Lambda_{\mathbf{x}}^{\frac{1}{2}} \nonumber \\
	&= \Lambda_{\mathbf{x}}^{\frac{1}{2}}  
	\begin{bmatrix} \sigma_{\mathbf{y} \left( 1 \right)}^2 & \sigma_{\mathbf{y} \left( 1 \right) \mathbf{y} \left( 2 \right)}^2 & \cdots & \sigma_{\mathbf{y} \left( 1 \right) \mathbf{y} \left( m \right)}^2 \\
	\sigma_{\mathbf{y} \left( 1 \right)\mathbf{y} \left( 2 \right)}^2 & \sigma_{\mathbf{y} \left( 2 \right)}^2 & \cdots & \sigma_{\mathbf{y} \left( 2 \right)\mathbf{y} \left( m \right)}^2 \\
	\vdots & \vdots & \ddots & \vdots \\
	\sigma_{\mathbf{y} \left( 1 \right)\mathbf{y} \left( m \right)}^2 & \sigma_{\mathbf{y} \left( 2 \right) \mathbf{y} \left( m \right)}^2 & \cdots & \sigma_{\mathbf{y} \left( m \right)}^2 	\end{bmatrix} \Lambda_{\mathbf{x}}^{\frac{1}{2}} \nonumber \\
	&=   
	\begin{bmatrix} \sigma_{\mathbf{x} \left( 1 \right)}^2 \sigma_{\mathbf{y} \left( 1 \right)}^2 & \cdots & \sigma_{\mathbf{x} \left( 1 \right)} \sigma_{\mathbf{x} \left( m \right)} \sigma_{\mathbf{y} \left( 1 \right) \mathbf{y} \left( m \right)}^2 \\
	\vdots & \ddots & \vdots \\
	\sigma_{\mathbf{x} \left( 1 \right)} \sigma_{\mathbf{x} \left( m \right)}\sigma_{\mathbf{y} \left( 1 \right)\mathbf{y} \left( m \right)}^2 &  \cdots &
	\sigma_{\mathbf{x} \left( m \right)}^2 \sigma_{\mathbf{y} \left( m \right)}^2 	\end{bmatrix} 
	,\nonumber
	\end{flalign} 
	where
	\begin{flalign}
	\Lambda_{\mathbf{x}}^{\frac{1}{2}}
	&= \left[ \mathrm{diag} \left( \sigma_{\mathbf{x} \left( 1 \right)}^2, \ldots, \sigma_{\mathbf{x} \left( m \right)}^2 \right) \right]^{\frac{1}{2}} \nonumber \\
	&= \mathrm{diag} \left( \sqrt{\sigma_{\mathbf{x} \left( 1 \right)}^2}, \ldots, \sqrt{\sigma_{\mathbf{x} \left( m \right)}^2} \right)
	\nonumber \\
	&= \mathrm{diag} \left( \sigma_{\mathbf{x} \left( 1 \right)}, \ldots, \sigma_{\mathbf{x} \left( m \right)} \right). \nonumber
	\end{flalign}
	It is clear that $\Lambda_{\mathbf{x}}^{\frac{1}{2}}  \Sigma_{\mathbf{y}}  \Lambda_{\mathbf{x}}^{\frac{1}{2}}$ is positive semi-definite.
	On the other hand, denote the eigenvalues of $\Lambda_{\mathbf{x}}^{\frac{1}{2}}  \Sigma_{\mathbf{y}}  \Lambda_{\mathbf{x}}^{\frac{1}{2}}$ as $\lambda_1, \ldots, \lambda_m$. Then,
	\begin{flalign}
	\tr \left[ \left( \Lambda_{\mathbf{x}}^{\frac{1}{2}}  \Sigma_{\mathbf{y}}  \Lambda_{\mathbf{x}}^{\frac{1}{2}} \right)^{\frac{1}{2}} \right] = \sum_{i=1}^{m} \lambda_i^{\frac{1}{2}}
	,\nonumber
	\end{flalign} 
	and it is known from Proposition~\ref{half} that
	\begin{flalign}
	\sum_{i=1}^{m} \lambda_i^{\frac{1}{2}} \leq \sum_{i=1}^{m}  \left[ \sigma_{\mathbf{x} \left( i \right)}^2 \sigma_{\mathbf{y} \left( i \right)}^2 \right]^{\frac{1}{2}}
	,\nonumber
	\end{flalign}  
	where equality holds if $\Lambda_{\mathbf{x}}^{\frac{1}{2}}  \Sigma_{\mathbf{y}}  \Lambda_{\mathbf{x}}^{\frac{1}{2}}$ is a diagonal matrix, i.e., 
	\begin{flalign}
	\sigma_{\mathbf{x} \left( i \right)} \sigma_{\mathbf{x} \left( j \right)} \sigma_{\mathbf{y} \left( i \right) \mathbf{y} \left( j \right) }^2
	= 0 \nonumber
	\end{flalign}
	for any $i, j = 1, \ldots, m; i \neq j$. Since $\Sigma_{\mathbf{x}}$ is positive definite (and hence $\sigma_{\mathbf{x} \left( i \right)} > 0,~\forall i = 1, \ldots, m$), this is equivalent to
	\begin{flalign}
	\sigma_{\mathbf{y} \left( i \right) \mathbf{y} \left( j \right) }^2
	= 0 \nonumber
	\end{flalign}
	for any $i, j = 1, \ldots, m; i \neq j$, i.e., $\Sigma_{\mathbf{y}}$ is diagonal as
	\begin{flalign}
	\Sigma_{\mathbf{y}}
	= \Lambda_{\mathbf{y}}
	= \mathrm{diag} \left( \sigma_{\mathbf{y} \left( 1 \right)}^2, \ldots, \sigma_{\mathbf{y} \left( m \right)}^2 \right). \nonumber
	\end{flalign}
	This completes the proof.
\end{proof}


Accordingly, we may examine the implications of Theorem~\ref{t1} for the Gelbrich bound, as shown in the following corollary; it is worth mentioning that herein $\mathbf{x}$ is not necessarily elliptical in the first place.

\begin{corollary} \label{c1}
		Consider $m$-dimensional random vectors $\mathbf{x}$ and $\mathbf{y}$ with positive definite covariance matrices $\Sigma_{\mathbf{x}}$ and $\Sigma_{\mathbf{y}}$, respectively. Suppose that $\Sigma_{\mathbf{x}}$ is a diagonal matrix, i.e.,
	\begin{flalign}
	\Sigma_{\mathbf{x}}
	= \Lambda_{\mathbf{x}}
	= \mathrm{diag} \left( \sigma_{\mathbf{x} \left( 1 \right)}^2, \ldots, \sigma_{\mathbf{x} \left( m \right)}^2 \right).
	\end{flalign}
	Denote the diagonal terms of $\Sigma_{\mathbf{y}}$ by $\sigma_{\mathbf{y} \left( 1 \right)}^2, \ldots, \sigma_{\mathbf{y} \left( m \right)}^2$, whereas $\Sigma_{\mathbf{y}}$ is not necessarily a diagonal matrix.
	Then,
	\begin{flalign} \label{b2}
	\mathcal{W}_2 \left( p_{\mathbf{x}} ; p_{\mathbf{y}} \right)
	& \geq \sqrt{\sum_{i=1}^{m} \left\{  \sigma_{\mathbf{x} \left( i \right)}^2 + \sigma_{\mathbf{y} \left( i \right)}^2 - 2 \left[ \sigma_{\mathbf{x} \left( i \right)}^2 \sigma_{\mathbf{y} \left( i \right)}^2 \right]^{\frac{1}{2}} \right\}} \nonumber \\
	& = \sqrt{\sum_{i=1}^{m} \left\{  \left[ \sigma_{\mathbf{x} \left( i \right)}^2  \right]^{\frac{1}{2}}
		- \left[  \sigma_{\mathbf{y} \left( i \right)}^2 \right]^{\frac{1}{2}} \right\}^2 }.
	\end{flalign}
\end{corollary}

\vspace{3mm}

\begin{proof}
	It is known from the proof of Theorem~\ref{t1} that
	\begin{flalign}
	&\sqrt{ \tr \left[ \Sigma_{\mathbf{x}} + \Sigma_{\mathbf{y}} - 2 \left( \Sigma_{\mathbf{x}}^{\frac{1}{2}}  \Sigma_{\mathbf{y}}  \Sigma_{\mathbf{x}}^{\frac{1}{2}} \right)^{\frac{1}{2}} \right] } \nonumber \\
	&~~~~ \geq \sqrt{\sum_{i=1}^{m} \left\{  \sigma_{\mathbf{x} \left( i \right)}^2 + \sigma_{\mathbf{y} \left( i \right)}^2 - 2 \left[ \sigma_{\mathbf{x} \left( i \right)}^2 \sigma_{\mathbf{y} \left( i \right)}^2 \right]^{\frac{1}{2}} \right\}} \nonumber \\
	&~~~~ = \sqrt{\sum_{i=1}^{m} \left\{  \left[ \sigma_{\mathbf{x} \left( i \right)}^2  \right]^{\frac{1}{2}}
		- \left[  \sigma_{\mathbf{y} \left( i \right)}^2 \right]^{\frac{1}{2}} \right\}^2 }. \nonumber
	\end{flalign}
	As such, \eqref{b2} follows directly from the Gelbrich bound.
\end{proof}

More generally, we may consider the cases when $\Sigma_{\mathbf{x}}$ is not necessarily diagonal, i.e., when $\mathbf{x}$ is not necessarily independent element-wise.

\begin{theorem} \label{t2}
	Consider $m$-dimensional random vectors $\mathbf{x}$ and $\mathbf{y}$ with positive definite covariance matrices $\Sigma_{\mathbf{x}}$ and $\Sigma_{\mathbf{y}}$, respectively. Suppose that $\mathbf{x}$ is elliptically distributed with density generator $g_\mathbf{x} \left( \mathbf{u} \right)$, whereas $\mathbf{y}$ is not necessarily elliptical. Denote the eigen-decomposition of $\Sigma_{\mathbf{x}}$ as
	\begin{flalign}
	\Sigma_{\mathbf{x}}
	= U_{\mathbf{x}} \Lambda_{\mathbf{x}} U_{\mathbf{x}}^{\mathrm{T}},
	\end{flalign}
	where
	\begin{flalign}
	\Lambda_{\mathbf{x}}
	= \mathrm{diag} \left( \lambda_1, \ldots, \lambda_m \right).
	\end{flalign}
	Meanwhile, denote
	\begin{flalign}
	\overline{\Sigma}_{\mathbf{y}}
	= U_{\mathbf{x}}^{\mathrm{T}} \Sigma_{\mathbf{y}} U_{\mathbf{x}},
	\end{flalign}
	and denote the diagonal terms of $\overline{\Sigma}_{\mathbf{y}}$ by $\overline{\sigma}_{\mathbf{y} \left( 1 \right)}^2, \ldots, \overline{\sigma}_{\mathbf{y} \left( m \right)}^2$, whereas $\overline{\Sigma}_{\mathbf{y}}$ is not necessarily diagonal; note also that $\Sigma_{\mathbf{y}}$ is not necessarily diagonal in the first place.
	Then,
	\begin{flalign}
	\mathcal{W}_2 \left( p_{\mathbf{x}} ; p_{\mathbf{y}} \right)
	& \geq \sqrt{\sum_{i=1}^{m} \left\{ \lambda_i + \overline{\sigma}_{\mathbf{y} \left( i \right)}^2 - 2 \left[ \lambda_i \overline{\sigma}_{\mathbf{y} \left( i \right)}^2 \right]^{\frac{1}{2}} \right\}} \nonumber \\
	& = \sqrt{\sum_{i=1}^{m} \left\{   \lambda_i^{\frac{1}{2}}
		- \left[  \overline{\sigma}_{\mathbf{y} \left( i \right)}^2 \right]^{\frac{1}{2}} \right\}^2 },
	\end{flalign}
	where equality holds if $U_{\mathbf{x}}^{\mathrm{T}} \mathbf{y}$ is elliptical with the same density generator as $\mathbf{x}$, i.e., $g_\mathbf{y} \left( \mathbf{u} \right) = g_\mathbf{x} \left( \mathbf{u} \right)$, while $\Sigma_{\mathbf{y}}$ is given by
	\begin{flalign}
	\Sigma_{\mathbf{y}}
	= U_{\mathbf{x}} \left[ \mathrm{diag} \left( \overline{\sigma}_{\mathbf{y} \left( 1 \right)}^2, \ldots, \overline{\sigma}_{\mathbf{y} \left( m \right)}^2 \right) \right]   U_{\mathbf{x}}^{\mathrm{T}}.
	\end{flalign}
\end{theorem}

\vspace{3mm}

\begin{proof}
	Note first that when \begin{flalign}
	\Sigma_{\mathbf{x}}
	= U_{\mathbf{x}} \Lambda_{\mathbf{x}} U_{\mathbf{x}}^{\mathrm{T}}, \nonumber
	\end{flalign}
	it can be verified that
	\begin{flalign}
	\Sigma_{\mathbf{x}}^{\frac{1}{2}}
	= \left( U_{\mathbf{x}} \Lambda_{\mathbf{x}} U_{\mathbf{x}}^{\mathrm{T}} \right)^{\frac{1}{2}}
	= U_{\mathbf{x}} \Lambda_{\mathbf{x}}^{\frac{1}{2}} U_{\mathbf{x}}^{\mathrm{T}}, \nonumber
	\end{flalign}
	and hence
	\begin{flalign}
	&\sqrt{ \tr \left[ \Sigma_{\mathbf{x}} + \Sigma_{\mathbf{y}} - 2 \left( \Sigma_{\mathbf{x}}^{\frac{1}{2}}  \Sigma_{\mathbf{y}}  \Sigma_{\mathbf{x}}^{\frac{1}{2}} \right)^{\frac{1}{2}} \right] } \nonumber \\
	& = \sqrt{ \tr \left[ U_{\mathbf{x}} \Lambda_{\mathbf{x}} U_{\mathbf{x}}^{\mathrm{T}} + \Sigma_{\mathbf{y}} - 2 \left( U_{\mathbf{x}} \Lambda_{\mathbf{x}}^{\frac{1}{2}} U_{\mathbf{x}}^{\mathrm{T}} \Sigma_{\mathbf{y}}  U_{\mathbf{x}} \Lambda_{\mathbf{x}}^{\frac{1}{2}} U_{\mathbf{x}}^{\mathrm{T}}\right) \right] } \nonumber \\
	& = \sqrt{ \tr \left\{ U_{\mathbf{x}} \left[  \Lambda_{\mathbf{x}} + U_{\mathbf{x}}^{\mathrm{T}}  \Sigma_{\mathbf{y}} U_{\mathbf{x}} - 2 \left(  \Lambda_{\mathbf{x}}^{\frac{1}{2}} U_{\mathbf{x}}^{\mathrm{T}} \Sigma_{\mathbf{y}}  U_{\mathbf{x}} \Lambda_{\mathbf{x}}^{\frac{1}{2}} \right) \right]  U_{\mathbf{x}}^{\mathrm{T}} \right\} } \nonumber \\
	& = \sqrt{ \tr \left\{ \left[  \Lambda_{\mathbf{x}} + U_{\mathbf{x}}^{\mathrm{T}}  \Sigma_{\mathbf{y}} U_{\mathbf{x}} - 2 \left(  \Lambda_{\mathbf{x}}^{\frac{1}{2}} U_{\mathbf{x}}^{\mathrm{T}} \Sigma_{\mathbf{y}}  U_{\mathbf{x}} \Lambda_{\mathbf{x}}^{\frac{1}{2}} \right) \right] U_{\mathbf{x}}^{\mathrm{T}}U_{\mathbf{x}}   \right\} } \nonumber \\
	& = \sqrt{ \tr \left[  \Lambda_{\mathbf{x}} + U_{\mathbf{x}}^{\mathrm{T}} \Sigma_{\mathbf{y}}  U_{\mathbf{x}} - 2 \left(  \Lambda_{\mathbf{x}}^{\frac{1}{2}} U_{\mathbf{x}}^{\mathrm{T}} \Sigma_{\mathbf{y}}  U_{\mathbf{x}} \Lambda_{\mathbf{x}}^{\frac{1}{2}} \right) \right] } \nonumber \\
	& = \sqrt{ \tr \left[  \Lambda_{\mathbf{x}} + \overline{\Sigma}_{\mathbf{y}} - 2 \left(  \Lambda_{\mathbf{x}}^{\frac{1}{2}} \overline{\Sigma}_{\mathbf{y}} \Lambda_{\mathbf{x}}^{\frac{1}{2}} \right) \right] }, \nonumber
	\end{flalign}
	where 
	\begin{flalign}
	\overline{\Sigma}_{\mathbf{y}}
	= U_{\mathbf{x}}^{\mathrm{T}} \Sigma_{\mathbf{y}} U_{\mathbf{x}}. \nonumber
	\end{flalign}
	Note also that $\overline{\Sigma}_{\mathbf{y}}$ is essentially the covariance of $U_{\mathbf{x}}^{\mathrm{T}} \mathbf{y}$.
	Then, it is known from the proof of Theorem~\ref{t1} that
	\begin{flalign}
	&\sqrt{ \tr \left[  \Lambda_{\mathbf{x}} + \overline{\Sigma}_{\mathbf{y}} - 2 \left(  \Lambda_{\mathbf{x}}^{\frac{1}{2}} \overline{\Sigma}_{\mathbf{y}} \Lambda_{\mathbf{x}}^{\frac{1}{2}} \right) \right] } \nonumber \\
	&~~~~ \geq \sqrt{\sum_{i=1}^{m} \left\{ \lambda_i + \overline{\sigma}_{\mathbf{y} \left( i \right)}^2 - 2 \left[ \lambda_i \overline{\sigma}_{\mathbf{y} \left( i \right)}^2 \right]^{\frac{1}{2}} \right\}} \nonumber \\
	&~~~~ = \sqrt{\sum_{i=1}^{m} \left\{  \lambda_i^{\frac{1}{2}}
		- \left[  \overline{\sigma}_{\mathbf{y} \left( i \right)}^2 \right]^{\frac{1}{2}} \right\}^2 }, \nonumber
	\end{flalign}
	where equality holds if $U_{\mathbf{x}}^{\mathrm{T}} \mathbf{y}$ is elliptical with the same density generator as $\mathbf{x}$, i.e., $g_\mathbf{y} \left( \mathbf{u} \right) = g_\mathbf{x} \left( \mathbf{u} \right)$, while $\overline{\Sigma}_{\mathbf{y}}$ is diagonal as
	\begin{flalign}
	\overline{\Sigma}_{\mathbf{y}}
	= \mathrm{diag} \left( \overline{\sigma}_{\mathbf{y} \left( 1 \right)}^2, \ldots, \overline{\sigma}_{\mathbf{y} \left( m \right)}^2 \right), \nonumber
	\end{flalign}
	or equivalently,
	\begin{flalign}
	\Sigma_{\mathbf{y}}
	= U_{\mathbf{x}} 
	\overline{\Sigma}_{\mathbf{y}} U_{\mathbf{x}}^{\mathrm{T}}
	= U_{\mathbf{x}} \left[ \mathrm{diag} \left( \overline{\sigma}_{\mathbf{y} \left( 1 \right)}^2, \ldots, \overline{\sigma}_{\mathbf{y} \left( m \right)}^2 \right) \right]   U_{\mathbf{x}}^{\mathrm{T}}. \nonumber
	\end{flalign}
	This concludes the proof.
\end{proof}

Correspondingly, we may examine what Theorem~\ref{t2} implicates for the Gelbrich bound for the general case where  $\mathbf{x}$ is not necessarily elliptical nor independent element-wise.

\begin{corollary} \label{c2}
	Consider $m$-dimensional random vectors $\mathbf{x}$ and $\mathbf{y}$ with positive definite covariance matrices $\Sigma_{\mathbf{x}}$ and $\Sigma_{\mathbf{y}}$, respectively. Denote the eigen-decomposition of $\Sigma_{\mathbf{x}}$ as
	\begin{flalign}
	\Sigma_{\mathbf{x}}
	= U_{\mathbf{x}} \Lambda_{\mathbf{x}} U_{\mathbf{x}}^{\mathrm{T}},
	\end{flalign}
	where
	\begin{flalign}
	\Lambda_{\mathbf{x}}
	= \mathrm{diag} \left( \lambda_1, \ldots, \lambda_m \right).
	\end{flalign}
	Meanwhile, denote
	\begin{flalign}
	\overline{\Sigma}_{\mathbf{y}}
	= U_{\mathbf{x}}^{\mathrm{T}} \Sigma_{\mathbf{y}} U_{\mathbf{x}},
	\end{flalign}
	and denote the diagonal terms of $\overline{\Sigma}_{\mathbf{y}}$ by $\overline{\sigma}_{\mathbf{y} \left( 1 \right)}^2, \ldots, \overline{\sigma}_{\mathbf{y} \left( m \right)}^2$, whereas $\overline{\Sigma}_{\mathbf{y}}$ is not necessarily a diagonal matrix.
	Then,
	\begin{flalign}
	\mathcal{W}_2 \left( p_{\mathbf{x}} ; p_{\mathbf{y}} \right)
	& \geq \sqrt{\sum_{i=1}^{m} \left\{ \lambda_i + \overline{\sigma}_{\mathbf{y} \left( i \right)}^2 - 2 \left[ \lambda_i \overline{\sigma}_{\mathbf{y} \left( i \right)}^2 \right]^{\frac{1}{2}} \right\}} \nonumber \\
	& = \sqrt{\sum_{i=1}^{m} \left\{   \lambda_i^{\frac{1}{2}}
		- \left[  \overline{\sigma}_{\mathbf{y} \left( i \right)}^2 \right]^{\frac{1}{2}} \right\}^2 }.
	\end{flalign}
\end{corollary}

\vspace{3mm}

\section{Conclusion}

We have presented a property of the $\mathcal{W}_2$ Wasserstein distance: Independent elliptical distributions minimize their $\mathcal{W}_2$ Wasserstein distance from given independent elliptical distributions with the same density generators. We have also examined the implications of this property in the Gelbrich bound when the distributions are not necessarily elliptical, while generalizing the results to the cases when the distributions are not independent. It might be interesting to further examine the implications of this property.

\balance

\bibliographystyle{IEEEtran}
\bibliography{references}

\end{document}